\documentclass[10pt,twoside]{siamart1116}
\usepackage[english]{babel}
\usepackage{graphicx,epstopdf,epsfig}
\usepackage{amsfonts,epsfig,fancyhdr,graphics, hyperref,amsmath,amssymb}

\newcommand{\HE}{Namie of Handling Editor}
\newcommand{\DoS}{Month/Day/Year}
\newcommand{\DoA}{Month/Day/Year}
\newcommand{\CA}{Name of Corresponding Author}

\newcommand{\Names}{Paola Ferrari, Isabella Furci, and Stefano Serra--Capizzano}
\newcommand{\Title}{Multilevel symmetrized Toeplitz structures and spectral distribution results for the related matrix-sequences}

\newtheorem{example}[theorem]{Example}

\newcommand{\E}{\mathrm{e}}
 
\renewcommand{\theequation}{\thesection.\arabic{equation}}
\renewtheorem{theorem}{Theorem}[section]

\begin{document}

\bibliographystyle{plain}

%  Leave these commented lines here
%\input{ELAheader-template.tex}
% ELA insert correct page number
\setcounter{page}{1}

\thispagestyle{empty}

%Insert the title of the paper
 \title{\Title\thanks{Received
 by the editors on \DoS.
 Accepted for publication on \DoA. 
 Handling Editor: \HE. Corresponding Author: \CA}}

\author{
Paola Ferrari \thanks{ Department of Science and high Technology, University of Insubria, Como, 22100, Italy (pferrari@uninsubria.it). Supported by INdAM Research group GNCS.}
% Remember to put \and between any two authors
\and
Isabella Furci \thanks{Department of Mathematics and Informatics, University of Wuppertal, Wuppertal,
42119, Germany (furci@uni-wuppertal.de). Supported by INdAM Research group GNCS.}
\and  
Stefano Serra--Capizzano \thanks{ Department of Humanities and Innovation, University of Insubria, Como, 22100, Italy (s.serracapizzano@uninsubria.it).  Supported by INdAM Research group GNCS.}}

\markboth{\Names}{\Title}

\maketitle

\begin{abstract}
In recent years, motivated by computational purposes, the singular value and spectral features of the symmetrization of Toeplitz matrices  generated by a Lebesgue integrable function have been studied. Indeed, under the assumptions that $f$ belongs to $L^1([-\pi,\pi])$ and it has real Fourier coefficients, the spectral and singular value distribution of the matrix-sequence $\{Y_nT_n[f]\}_n$ has been identified, where $n$ is the matrix-size, $Y_n$ is the anti-identity matrix, and $T_n[f]$ is the Toeplitz matrix generated by $f$.
In this note, we consider the multilevel Toeplitz matrix $T_{\bf n}[f]$  generated by $f\in L^1([-\pi,\pi]^k)$, $\bf n$ being a multi-index identifying the matrix-size, and we prove spectral and singular value distribution results for the matrix-sequence $\{Y_{\bf n}T_{\bf n}[f]\}_{\bf n}$ with $Y_{\bf n}$ being the corresponding tensorization of the anti-identity matrix. 
\end{abstract}


% Sample article for the Electronic Journal of Linear Algebra

%%%%%%%%%%%%%%%%%%%%%%%%%%%%%%%%%%%%%%%

%%%%%%%%%%%%%%%%%%%%%%%%%%%%%%%%%%%%%%%%%%%%%%%%%%%%%%%%%%%%%

\section{Introduction}

Spectral and singular value distribution results \cite{MR952991,FasinoTilli,MR890515,SerraLibro1,SerraLibro2,MR1671591,Tyrtyshnikov19961,MR1481397} of structured matrix-sequences represent one anong the key ingredients in the design and in the convergence analysis of several  well-known (preconditioned) iterative methods \cite{Saad, Hackbusch}. In many contexts, symmetry is a particularly desirable property for a matrix when we want to solve an associated linear system with iterative methods. Hence,  from the original work by Pestana and Wathen \cite{doi:10.1137/140974213}, symmetrization procedures combined with various preconditioning techniques have been introduced and studied for the very purpose of developing a competitive method for the solution of real non-symmetric structured systems. 
In particular the singular value and spectral features of the symmetrization of Toeplitz matrices  generated by a Lebesgue integrable function have been recently  discussed and exploited in several settings. Indeed, under the assumptions that $f$ belongs to $L^1([-\pi,\pi])$ and it has real Fourier coefficients,  the spectral and singular value distribution of the matrix-sequence $\{Y_nT_n[f]\}_n$ has been studied, where $Y_n$ is the anti-identity matrix and $T_n[f]$ is the Toeplitz matrix generated by $f$  \cite{FFHMS, mazza-pestana}. Several extensions of the latter result have been also treated. For example the generalization in the context of block structures is treated in \cite{FFHMS}, i.e. assuming $f$ a matrix-valued function, while the spectral distribution of  matrix-sequences of the form $\{h(T_n[f])\}_n$, with $h$ being an analytic function, is studied in \cite{FBS}. 

The  purpose of this paper is to extend the result concerning the eigenvalue and singular value distributions of the unilevel matrix-sequence $\{Y_nT_n[f]\}_{{n}}$ to the symmetrization of multilevel matrix-sequences of the form $\{T_{\mathbf{n}}[f]\}_{\mathbf{n}}$, where $f$ is a $k$-variate function $f\in L^1([-\pi,\pi]^k)$.   The proof of the main Theorem of the paper is based on the relation between  a Toeplitz matrix and its generating function $f$ and on the notion of approximating class of sequences (a.c.s.), as it has been done in the simpler unilevel case in \cite{FFHMS}: it should be noted that the approach and the proof techniques in \cite{mazza-pestana} and in the very recent \cite{mazza-pestana-m} are different since the authors employ the powerful Generalized Locally Toeplitz (GLT) matrix-sequences technology, which in turn heavily relies on the a.c.s. notion. These preliminary concepts are introduced in Section \ref{section:preliminaries} in the general $k$-level setting. In Section \ref{section:main} we give our main result on the asymptotic distributions of $\{Y_{\mathbf{n}}T_{\mathbf{n}}[f]\}_{\mathbf{n}}$, first considering the case where $f$ is a trigonometric polynomial, then extending the result to $f\in L^1([-\pi,\pi]^k)$, $k>1$. In addition, Section \ref{section:further} is devoted to further results and observations, where we dedicate particular attention to the case where $f$ is a $k$-variate separable function. In Section \ref{section:numerics} we collect relevant experiments, showing the numerical validity and accuracy of our theoretical findings. Finally,  Section \ref{sec:conclusion} is devoted to conclusions and open problems.

\section{Preliminaries on Toeplitz matrices}\label{section:preliminaries}

Here we define the multi-index $\mathbf{n}=(n_1,n_2,\dots,n_k)$ where each $n_j$ is a positive integer.  When writing the expression ${{\bf n}\to\infty}$ we mean that every component of the vector ${\bf n}$ tends to infinity i.e. $\min_{1\le j\le k} n_j\to \infty$.

Let $f:$~$[-\pi,\pi]^k\to \mathbb{C}$ be a function belonging to $L^{1}([-\pi,\pi]^k)$, and periodically extended to $\mathbb{R}^k$. We define $T_{\mathbf{n}}[f] $ the multilevel Toeplitz matrix 
of dimensions $N(\mathbf{n})\times N(\mathbf{n})$, with $N(\mathbf{n})= n_1n_2 \dots n_k$, as follows

\begin{equation*}
T_{\bf n}[f] =\sum_{|j_1|<n_1}\ldots \sum_{|j_k|<n_k} (J_{n_1}^{j_1} \otimes \cdots\otimes J_{n_k}^{j_k}){\hat{f}_{\mathbf{j}}}, \quad \mathbf{j}=(j_1,j_2,\dots,j_k)\in \mathbb{Z}^k.
\end{equation*}
In the latter, the quantities
 \[
\hat{f}_{\mathbf{j}}=\frac{1}{(2\pi)^k}\int_{[-\pi,\pi]^k}f(\boldsymbol{\theta}){\rm e}^{\iota \left\langle { \bf j},\boldsymbol{\theta}\right\rangle}\, {\rm d}\boldsymbol{\theta},\]
 with $\left\langle { \bf j},\boldsymbol{\theta}\right\rangle=\sum_{t=1}^kj_t\theta_t$, $\iota^2=-1$, are the Fourier coefficients of $f$ and 
  $ J^{j}_{n}$  is the $n \times n$ matrix whose $(l,h)$-th entry equals 1 if $(l-h)=j$ and $0$ otherwise.
 
 \begin{lemma}\cite{SerraLibro2}
\label{lemm:tensor_prod}
Let $f_1 , \dots, f_k \in L^1([-\pi,\pi])$,  $\mathbf{n}=(n_1,n_2, \dots, n_k) \in \mathbb{N}^k$. Then,
\begin{equation*}
 T_{ n_1} [ f_1 ] \otimes \dots \otimes T_{n_k} [ f_k ]= T_{\mathbf{n}} [f_1 \otimes \dots \otimes f_k ],
\end{equation*}
where the Fourier coefficients of $f_1 \otimes \dots \otimes f_k$ are given by
\begin{equation*}
( f_1 \otimes \dots \otimes f_k)_{\mathbf{j}} = ( f_1 )_{j_1} \dots ( f_k )_{j_k}, \quad \mathbf{j} \in \mathbb{Z}^k.
\end{equation*} 
\end{lemma}

Throughout the paper we indicate by $\{T_{\bf n}[f]\}_{\mathbf{n}}$ the matrix-sequence whose elements are the matrices $T_{\bf n}[f]$. The function $f$ is called the \emph{generating function} of $T_\mathbf{n}[f]$.

 If $f$ is complex-valued, then $T_\mathbf{n}[f]$ is non-Hermitian for all sufficiently large $\mathbf{n}$. Conversely, if $f$ is real-valued, then $T_\mathbf{n}[f]$ is Hermitian for all $\mathbf{n}$. If $f$ is real-valued and nonnegative, but not identically zero almost everywhere, then $T_\mathbf{n}[f]$ is Hermitian positive definite for all $\mathbf{n}$. If $f$ is real-valued and even, $T_\mathbf{n}[f]$ is symmetric for all $\mathbf{n}$ \cite{MR2108963,MR2376196}.

%\textcolor{red}{accorciare un po' il discorso?}
	
	The singular value and spectral distribution of Toeplitz matrix-sequences has been well studied in the past few decades. Ever since Szeg{\H{o}} in \cite{MR890515} showed that the eigenvalues of the Toeplitz matrix $T_n[f]$ generated by real-valued $f\in L^{\infty}([-\pi,\pi])$ are asymptotically distributed as $f$. Moreover, Avram and Parter \cite{MR952991,MR851935} proved that the singular values of $T_n[f]$ are distributed as $|f|$ for a complex-valued $f\in L^{\infty}([-\pi,\pi])$. Tyrtyshnikov and Zamarashkin \cite{Tyrtyshnikov19961,MR1258226,MR1481397}and, independently, Tilli \cite{MR1671591} later extended the spectral and singular value theorems to Toeplitz matrices $T_n[f]$ generated by functions $f\in L^1([-\pi,\pi])$. Recently, Garoni, Serra--Capizzano, and Vassalos \cite{MR3399336} provided the same theorem \textcolor{black}{in the unilevel case} based on the theory of Generalized Locally Toeplitz (GLT) sequences \cite{SerraLibro1}. \textcolor{black}{As for} the changes in the singular value and spectral distribution of Toeplitz matrix-sequences after certain matrix operations \textcolor{black}{that are related to our concerned problems, much work was done by Tyrtyshnikov and Serra--Capizzano in \cite{TYRTYSHNIKOV1994225,CAPIZZANO2001121,glt-1,glt-2}.}
	
	All these consideration suggest that  the candidate function which describes the asymptotic distributions of Toeplitz matrix-sequences is the generating function, but this is actually verified only under specific hypotheses. Indeed, if a Toeplitz matrix is not Hermitian, in general we cannot describe its spectral properties studying the generating function. However, the knowledge of the spectral and singular value information is crucial in the design proper and fast methods for the solution of Toeplitz systems. Then, the study of strategies and their properties that permit us to symmetrize Toeplitz linear systems is fundamental and convenient.
	
	\textcolor{black}{Throughout this work, we assume} that $f\in L^1([-\pi, \pi]^k)$ and is periodically extended to $\mathbb{R}^k$. Furthermore, we follow all standard notation and terminology introduced in \cite{SerraLibro1}: let $C_c(\mathbb{C})$ (or $C_c(\mathbb{R})$) be the space of complex-valued continuous functions defined on $\mathbb{C}$ (or $\mathbb{R}$) with bounded support and \textcolor{black}{let $\eta$ be a functional, i.e. any function defined on some vector space which takes values in $\mathbb{C}$. Also, if $g:D\subset \mathbb{R}^k \to \mathbb{K}$ ($\mathbb{R}$ or $\mathbb{C}$) is a measurable function defined on a set $D$ with $0<\mu_k(D)<\infty$, the functional $\eta_g$ is denoted such that
	\[
	\eta_g:C_c(\mathbb{K})\to \mathbb{C} \quad \text{and} \quad \eta_g(F)=\frac{1}{\mu_k(D)}\int_D F(g(\mathbf{x}))\,d\mathbf{x}.
	\]
\begin{definition}{\rm \cite[Definition 3.1]{SerraLibro1}(Singular value and eigenvalue distribution of a matrix-sequence)}\label{def:spectral_distribution}
Let $\{A_\mathbf{n}\}_{n}$ be a matrix-sequence.
\begin{enumerate}
\item We say that $\{A_\mathbf{n}\}_{\mathbf{n}}$ has an asymptotic singular value distribution described by a functional $\eta:C_c(\mathbb{R})\to \mathbb{C},$ and we write $\{A_\mathbf{n}\}_{\mathbf{n}} \sim_{\sigma}\eta,$ if
\[
\lim_{\mathbf{n}\to \infty} \frac{1}{N(\mathbf{n})}\sum_{j=1}^{N(\mathbf{n})}F(\sigma_j(A_\mathbf{n}))=\eta(F),\quad \forall F \in C_c(\mathbb{R}).
\] If $\eta=\eta_{|f|}$ for some measurable $f:D \subset \mathbb{R}^k \to \mathbb{C}$ defined on a set $D$ with $0<\mu_k(D)<\infty,$ we say that $\{A_\mathbf{n}\}_{\mathbf{n}}$ has an asymptotic singular value distribution described by $f$ and we write $\{A_\mathbf{n}\}_{\mathbf{n}} \sim_{\sigma} f.$% In this case, the function $f$ is referred to as the singular value symbol of the matrix-sequence $\{A_n\}_n$.
\item We say that $\{A_\mathbf{n}\}_{\mathbf{n}}$ has an asymptotic eigenvalue (or spectral) distribution described by a functional $\eta:C_c(\mathbb{R})\to \mathbb{C},$ and we write $\{A_\mathbf{n}\}_{\mathbf{n}} \sim_{\lambda}\eta,$ if
\[
\lim_{\mathbf{n}\to \infty} \frac{1}{N(\mathbf{n})}\sum_{j=1}^{N(\mathbf{n})}F(\lambda_j(A_\mathbf{n}))=\eta(F),\quad \forall F \in C_c(\mathbb{C}).
\] If $\eta=\eta_{f}$ for some measurable $f:D \subset \mathbb{R}^k \to \mathbb{C}$ defined on a set $D$ with $0<\mu_k(D)<\infty,$ we say that $\{A_\mathbf{n}\}_{\mathbf{n}}$ has an asymptotic eigenvalue (or spectral) distribution described by  $f$ and we write $\{A_\mathbf{n}\}_{\mathbf{n}} \sim_{\lambda} f.$% In this case, the function $f$ is referred to as the eigenvalue (or spectral) symbol of the matrix-sequence $\{A_n\}_n$.
		\end{enumerate}
	\end{definition}}

\textcolor{black}{In the following the generalized Szeg{\H{o}} theorem that describes the singular value and spectral distribution of Toeplitz sequences is given in the multivariate setting. We refer to \cite[Theorem 3.5]{SerraLibro2} for a proof that is based on the notion of approximating class of sequences given in Definition \ref{def:ACS}.}
		
\begin{theorem}\label{szego}
Let  $f\in L^1([-\pi, \pi]^k)$, with $k\ge 1$. Then, $$\left\{T_{\bf n}[f]\right\}_{\mathbf{n}}\sim_\sigma f.$$
Moreover if $f$ is a real-valued function, then
$$\left\{T_{\bf n}[f]\right\}_{{\mathbf{n}}}\sim_\lambda f.$$
\end{theorem}
Moreover, we introduce the following definitions and a key lemma in order to prove our main distribution results in the next {section}.
Regarding the employed norms we use the following notation: $\|\cdot\|$ denotes the spectral norm for matrices (that is the maximal singular value, also called Schatten-$\infty$ norm), $\|\cdot\|_1$ denotes the trace norm for matrices (that is the sum of all the  singular values, also called Schatten-$1$ norm), and $\|\cdot\|_{L^1}$ denotes the standard $L^1$ norm for functions.
	
	\begin{definition}{\rm \cite[Definition 2.6]{SerraLibro2}(Approximating class of sequences)}\label{def:ACS}
		Let $\{A_\mathbf{n}\}_{\mathbf{n}}$ be a matrix-sequence and let $\{\{B_{\mathbf{n},m}\}_{\mathbf{n}}\}_m$ be a sequence of matrix-sequences. We say that $\{\{B_{\mathbf{n},m}\}_{\mathbf{n}}\}_m$ is an \textit{approximating class of sequences (a.c.s)} for $\{A_\mathbf{n}\}_{\mathbf{n}}$ if the following condition is met: for every $m$ there exists $n_m$ such that, for $n \geq n_m$,
		\[
		A_\mathbf{n}=B_{\mathbf{n},m}+R_{\mathbf{n},m}+N_{\mathbf{n},m},
		\]
		\[
		\text{rank}~R_{\mathbf{n},m}\leq c(m)N(\mathbf{n}) \quad {\rm and} \quad \|N_{\mathbf{n},m}\|\leq\omega(m),
		\]
		where $n_m$, $c(m)$, and $\omega(m)$ depend only on $m$ and \[\lim_{m\to\infty}c(m)=\lim_{m\to\infty}\omega(m)=0.\]
	\end{definition}
	
	We use $\{B_{\mathbf{n},m}\}_{\mathbf{n}}\xrightarrow{\text{a.c.s.\ wrt\ $m$}}\{A_\mathbf{n}\}_{\mathbf{n}}$ to denote that $\{\{B_{\mathbf{n},m}\}_{\mathbf{n}}\}_m$ is an a.c.s for $\{A_\mathbf{n}\}_{\mathbf{n}}$.
	
	The following is a useful criterion to identify an a.c.s. without constructing the splitting present in Definition \ref{def:ACS}.
	\begin{theorem}\label{thm:acs_caratt}
	Let $\{A_\mathbf{n}\}_{\mathbf{n}}$  be a sequence of matrices, with $A_\mathbf{n}$ of size $N(\mathbf{n})$, let $\{\{B_{\mathbf{n},m}\}_{\mathbf{n}}\}_m$ be a sequence of matrix-sequences, with $B_{\mathbf{n},m}$ of size $N(\mathbf{n})$. Suppose that for every $m$ there exists $\mathbf{n}_m$ such that, for $  \mathbf{n}\ge\mathbf{n}_m$
	\[\|A_\mathbf{n}- B_{\mathbf{n},m}\|_{1}\le \epsilon(m)N(\mathbf{n}),\]

where $\lim_{m\to \infty}  \epsilon(m)= 0$. Then,
\[\{B_{\mathbf{n},m}\}_{\mathbf{n}}\xrightarrow{\text{a.c.s.\ wrt\ $m$}}\{A_\mathbf{n}\}_{\mathbf{n}}.\]
	\end{theorem}
\begin{lemma}{\rm \cite[Corollary 2.4]{SerraLibro2}}\label{lem:Corollary5.1}
		Let $\{A_\mathbf{n}\}_{\mathbf{n}}, \{B_{\mathbf{n},m}\}_{\mathbf{n}}$ be matrix-sequences and let $f,f_m:D \subset \mathbb{R}^k \to \mathbb{C}$ be measurable functions defined on a set $D$ with $0<\mu_k(D)<\infty$. Suppose that
		
		\begin{enumerate}
			\item $\{B_{\mathbf{n},m}\}_{\mathbf{n}} \sim_{\sigma}  f_m$ for every $m$,
			\item $\{B_{\mathbf{n},m}\}_{\mathbf{n}}\xrightarrow{\text{a.c.s.\ wrt\ $m$}}\{A_\mathbf{n}\}_{\mathbf{n}}$,
			\item $f_m \to f$ in measure.
		\end{enumerate}
		
		Then 
		\[
		\{A_\mathbf{n}\}_{\mathbf{n}} \sim_{\sigma}  f.
		\]
		
		Moreover, if the first assumption is replaced by $\{B_{\mathbf{n},m}\}_{\mathbf{n}}  \sim_{\lambda}  f_m$ for every $m$, given that the other two assumptions are left unchanged, and all the involved matrices are Hermitian, then  $\{A_\mathbf{n}\}_{\mathbf{n}} \sim_{\lambda}  f$.
	\end{lemma}

	In \cite{FFHMS} a useful asymptotic spectral result is provided in the  unilevel setting for matrix-sequences $\{Y_nT_n[f]\}_n$, where $Y_n$ is the anti-identity matrix.
\begin{theorem}\label{thm:main_old1d}	\cite{FFHMS}
	Let $Y_n$ be the $n\times n$ anti-identity matrix  and $T_n[f]$ be the  $n\times n$ Toeplitz matrix generated by a univariate Lebesgue integrable function $f\in L^1([-\pi,\pi])$.
	 If $f$ has real Fourier coefficients, then the matrix-sequence $\{Y_nT_n[f]\}_n$ is distributed in  eigenvalue sense as the function
			\begin{equation}\label{def:psi_unilevel}
			\psi_{|f|}(\theta)=\left\{
			\begin{array}{cc}
			|f|(\theta), & \theta\in [0,2\pi], \\
			-|f|(\theta+2\pi), & \theta\in [-2\pi,0) 
			\end{array}
			\right. .
			\end{equation}
			\end{theorem}
			
Our main goal is the generalization of the latter result for a $k$-variate $f\in L^1([-\pi,\pi]^k)$, $k>1$. First, in Theorem \ref{thm:singularvalue_multi}, we provide a precise description of the asymptotic singular value distribution of $\{Y_\mathbf{n}T_\mathbf{n}[f]\}_{\mathbf{n}} $, where the real Toeplitz matrix $T_\mathbf{n}[f]$ is generated by $f \in L^1([-\pi,\pi]^k)$. Then, in Section \ref{section:main} we exploit it to provide an elegant description of the asymptotic spectral distribution
of $\{Y_\mathbf{n}T_\mathbf{n}[f]\}_{\mathbf{n}} $.
	
	\begin{theorem}\label{thm:singularvalue_multi}
		Suppose $f \in L^1([-\pi,\pi]^k)$ with real Fourier coefficients and $Y_\mathbf{n} \in \mathbb{R}^{N(\mathbf{n}) \times N(\mathbf{n})}$ is the multilevel anti-identity matrix $Y_\mathbf{n}=Y_{n_1}\otimes \ldots \otimes Y_{n_k}=Y_{N(\mathbf{n})}$. Let $T_\mathbf{n}[f]\in \mathbb{R}^{N(\mathbf{n})\times N(\mathbf{n})}$ be the Toeplitz matrix generated by $f$. Then 
		\[
		\{Y_\mathbf{n}T_\mathbf{n}[f]\}_{\mathbf{n}}  \sim_{\sigma}  f.
		\] 
		\begin{proof}
Consider the (full) singular value decomposition of $T_\mathbf{n}[f]=U \Sigma V^*$, where $U,V$ are unitary matrices of size $N(\mathbf{n})$ and $\Sigma$ is the diagonal matrix containing the singular values $\sigma_1,\dots,\sigma_{N(\mathbf{n})}$ of $T_\mathbf{n}[f]$.
		 
We can write 
\[Y_{N(\mathbf{n})}T_{\mathbf{n}}[f]=\left(Y_{N(\mathbf{n})}U\right)\Sigma V^*.\]

Since $Y_{N(\mathbf{n})}$ is a unitary matrix, $Y_{N(\mathbf{n})}U$ is unitary and the previous formula is a singular value decomposition of $Y_{N(\mathbf{n})}T_{\mathbf{n}}[f]$.
Hence the sequence $\{Y_{N(\mathbf{n})}T_{\mathbf{n}}[f]\}_{\mathbf{n}}$ has the same singular value distribution of $\{T_{\mathbf{n}}[f]\}_{\mathbf{n}} $, which we know from Theorem \ref{szego}.
Consequently
\[\{Y_\mathbf{n}T_\mathbf{n}[f]\}_{\mathbf{n}}  \sim_{\sigma}  f,\]
and this completes the proof.
		 
		\end{proof}

	\end{theorem}

\section{Main results}\label{section:main}
	
	\textcolor{black}{In this section, we provide the main results on the spectral distribution of $\{Y_\mathbf{n} T_\mathbf{n}[f]\}_{\mathbf{n}} $.}
	
	First we report a general tool useful for the latter purpose and we define the function that will have a crucial role in the description of the spectrum of $Y_\mathbf{n} T_\mathbf{n}[f]$.
	\begin{definition}\label{def:psi}
	Given the vector $\mathbf{p}=[2\pi,2\pi,\dots,2\pi]^T\in \mathbb{R}^k$ and a function $g$ defined over $[0,2\pi]^k$, we define 
$\psi_{g}$ over $[-2\pi,0]^k\cup [0,2\pi]^k$ in the following manner
	\begin{equation}\label{eq:psi}
\psi_g(\boldsymbol{\theta})=\left\{
	\begin{array}{cc}
	g(\boldsymbol{\theta}), & \boldsymbol{\theta}\in [0,2\pi]^k, \\
	-g(\boldsymbol{\theta}+\mathbf{p}), & \boldsymbol{\theta}\in  [-2\pi,0]^k, \ \boldsymbol{\theta}\neq \mathbf{0}
	\end{array}.
	\right.\,
	\end{equation}
	
	\end{definition}

	\begin{theorem}\label{Lemm:symm}
		Suppose $n \in \mathbb{Z}$ and $A(n)\in \mathbb{C}^{ n\times n}$. Let $B_n \in \mathbb{C}^{2n \times 2n}$ be \textcolor{black}{Hermitian matrices} such that
		\[
		B_n=\left[\begin{array}{cc}
		O  & A(n) \\
		A(n)^H & O\end{array}\right],
		\]
		with $O$ being the square null matrices of size $n$.		If $\sigma_1,\dots, \sigma_n$ are the singular values of $A(n)$, then the eigenvalue of $B_n$ are given by $\pm  \sigma_j$, $j=1,\dots, n$.
	\end{theorem}
	
In the following we show that the spectral distribution of  $\{Y_\mathbf{n}  T_\mathbf{n}[f]\}_{\mathbf{n}} $ is described by $\psi_{|f|}(\boldsymbol{\theta})$ over the domain $[-2\pi,0]^k\cup [0,2\pi]^k$. In particular, in Theorem \ref{thm:main_pol} we prove that this holds for a trigonometric polynomial and in Theorem \ref{thm:main_L1} we extend the result to a generic $f\in L^1([-\pi,\pi]^k)$.
	\begin{theorem}\label{thm:main_pol}
		Suppose that $f$ is a $k$-variate trigonometric polynomial of degree $\mathbf{r}=(r_1,r_2,\dots,r_k)$ with real Fourier coefficients. Let $Y_\mathbf{n} \in \mathbb{R}^{N(\mathbf{n}) \times N(\mathbf{n})}$ be the anti-identity matrix $Y_\mathbf{n}=Y_{n_1}\otimes \ldots \otimes Y_{n_k}=Y_{N(\mathbf{n})}$ and let $T_\mathbf{n}[f]\in \mathbb{R}^{N(\mathbf{n}) \times N(\mathbf{n})}$ be the Toeplitz matrix generated by $f$. Then,
		\[
		\{Y_\mathbf{n}T_\mathbf{n}[f]\}_{\mathbf{n}}  \sim_{\lambda}  \psi_{|f|}
		\]
		over the domain $[-2\pi,0]^k\cup [0,2\pi]^k$, where $\psi_{|f|}$ is given as in Definition \ref{def:psi}.
	\end{theorem}
	
	\begin{proof}

%		We let $H_\nu[f,-]$ be the $\nu$-by-$\nu$ Hankel matrix generated by $f$ containing the Fourier coefficients from $a_{-1}$ in the position $(1,1)$ to $a_{-2\nu+1}$ in the position $(\nu,\nu)$. Analogously, we let $H_\nu[f,+]$ be the $\nu$-by-$\nu$ Hankel matrix generated by $f$ containing the Fourier coefficients from $a_{1}$ in the position $(1,1)$ to $a_{2\nu-1}$ in the position $(\nu,\nu)$.
		
First, we assume that we are in the case $\mathbf{n}=(n_1,n_2,\dots, n_k)$ with even $n_1=2\nu_1$.
		The trigonometric polynomial $f$ can be written in terms of its Fourier coefficients as
		\begin{equation*}
f(\boldsymbol{\theta})=\sum_{\mathbf{j}=-\mathbf{r}}^{\mathbf{r}}\hat{f}_{\mathbf{j}}\E^{\iota\left\langle {\mathbf{j}},\boldsymbol{\theta}\right\rangle},		
		\end{equation*}
where $\boldsymbol{\theta}=(\theta_1,\ldots,\theta_k)$, $\left\langle \mathbf{j},\boldsymbol{\theta}\right\rangle=\sum_{i=1}^kj_i\theta_i$.
		
		Consider the following $2\times 2$ block matrix-sequence  
		\begin{equation*}
	\{M_{\mathbf{n}}\}_{\mathbf{n}}=	\left\{
		\left[\begin{array}{cc}
		O & Y_{\mathbf{\tilde{n}}}T_{\mathbf{\tilde{n}}}[f] \\
		 Y_{\mathbf{\tilde{n}}}T_{\mathbf{\tilde{n}}}[f] & O\end{array}\right]
		\right\}_{\mathbf{n}},
		\end{equation*}	
		with blocks of dimension  $ N({\mathbf{\tilde{n}}})= \nu_1n_2 \dots n_k$, $ {\mathbf{\tilde{n}}}= (\nu_1,n_2, \dots, n_k)$.
		
	 Due to its particular structure, we can easily obtain the asymptotic eigenvalue distribution of  $\{M_{\mathbf{n}}\}_{{n}}$.
		Indeed Theorem \ref{Lemm:symm} implies that the eigenvalues of $M_{\mathbf{n}}$ are $\pm \sigma_j(Y_{\mathbf{\tilde{n}}}T_{\mathbf{\tilde{n}}}[f])=\pm \sigma_j( T_{\mathbf{\tilde{n}}}[f])$,
		$j=1,\ldots, N({{\mathbf{\tilde{n}}}})$. 
		
Consequently, we can conclude that 
		\begin{equation}\label{eq:distr_eig_Tn}
		\{M_\mathbf{n}\}_{\mathbf{n}} \sim_{\lambda}  \psi_{|f|}.
		\end{equation}
Concerning the matrix $Y_\mathbf{n}T_\mathbf{n}[f]$, we note that it can be written as a $2\times2$  block matrix with block of size $ N({\mathbf{\tilde{n}}})\times N({\mathbf{\tilde{n}}}) $. Indeed,
	\begin{equation}\label{eq:dec_hank_tep}
	\begin{split}
		Y_\mathbf{n}T_\mathbf{n}[f] &= \left[\begin{array}{cc}
		H^{(1)}_{\mathbf{\tilde{n}}} & Y_{\mathbf{\tilde{n}}}T_{\mathbf{\tilde{n}}}[f] \\
		 Y_{\mathbf{\tilde{n}}}T_{\mathbf{\tilde{n}}}[f] & H^{(2)}_{\mathbf{\tilde{n}}}\end{array}\right]= M_{\mathbf{n}}+\left[\begin{array}{cc}
		H^{(1)}_{\mathbf{\tilde{n}}} & O \\
		O & H^{(2)}_{\mathbf{\tilde{n}}}\end{array}\right],
	\end{split}
	\end{equation}
	where $H^{(1)}_{\mathbf{\tilde{n}}} $ and $H^{(2)}_{\mathbf{\tilde{n}}} $ are particular multilevel Hankel matrices of size $ N({\mathbf{\tilde{n}}})\times N({\mathbf{\tilde{n}}}) $.
	
	To conclude the proof we want to show that $\{Y_{\mathbf{n}}T_{\mathbf{n}}[f]\}_{\mathbf{n}}$,  $\{M_\mathbf{n}\}_{\mathbf{n}}$ and $\psi_{|f|}$ satisfy the hypotheses of Lemma \ref{lem:Corollary5.1}, 
	with $f_m\equiv \psi_{|f|}$, $\{A_{\mathbf{n}}\}_{\mathbf{n}}\equiv\{M_{\mathbf{n}}\}_{\mathbf{n}}$ and 
	$\{B_{\mathbf{n},m}\}_{\mathbf{n}}\equiv\{Y_{\mathbf{n}}T_{\mathbf{n}}[f]\}_{\mathbf{n}}$. That is, $\{Y_{\mathbf{n}}T_{\mathbf{n}}[f]\}_{\mathbf{n}}$ 
is a constant class of sequences, that is not depending on the variable $m$.
In particular, it is sufficient to verify that  
\[\{Y_{\mathbf{n}}T_{\mathbf{n}}[f]\}_{\mathbf{n}}\xrightarrow{\text{a.c.s.\ wrt\ $m$}}\{M_\mathbf{n}\}_{\mathbf{n}},\]
since we already have relation (\ref{eq:distr_eig_Tn}) and obviously $\{\psi_{|f|}\}_m\to \psi_{|f|}$ in measure.

Consequently, according to the Definition \ref{def:ACS}	it is sufficient to prove that the matrix-sequence
	\begin{equation*}
\{	E_{\mathbf{n}}\}_{\mathbf{n}}= \left\{\left[\begin{array}{cc}
		H^{(1)}_{\mathbf{\tilde{n}}} & O \\
		O & H^{(2)}_{\mathbf{\tilde{n}}}\end{array}\right]\right\}_{\mathbf{n}}
	\end{equation*}
%is a zero distributed sequence in  the singular value sense. 
%
%That is, the matrix-sequences $\{H^{(1)}_\mathbf{m}[f]\}_{\mathbf{m}}$ and $\{H^{(2)}_\mathbf{m}[f]\}_{\mathbf{m}}$ are zero distributed sequence in  the singular value sense.
% In the following we show the proof only for $\{H^{(1)}_\mathbf{m}[f]\}_{\mathbf{m}}$, since for $\{H^{(2)}_\mathbf{m}[f]\}_{\mathbf{m}}$ is analogous.
is such that for $\mathbf{n}\ge \mathbf{n}_m $
	\[
		\text{rank}~E_{\mathbf{n}}\leq c(m)N(\mathbf{n}),
		\]
where $c(m)$ depends only on $m$ and $\lim_{m\to\infty}c(m)=0.$

		The fact that the  $f$ is a trigonometric polynomial of degree $\mathbf{r}$ implies that
		
\begin{equation*}
{\rm rank}\,H^{(i)}_{\mathbf{\tilde{n}}}\le r_1n_2n_3\dots n_k, \quad i=1,2.
\end{equation*}
Then, for $\mathbf{n}\ge \mathbf{n}_m$,
\begin{equation*}
{\rm rank}(E_{\mathbf{n}})\le 2 r_1n_2n_3\dots n_k = \frac{2r_1m}{m}  n_2n_3\dots n_k.
\end{equation*}
and so, for $\mathbf{n}$ such that $n_1\ge 2r_1m$, we obtain 		
\begin{equation*}
{\rm rank}(E_{\mathbf{n}})\le \frac{1}{m}N(\mathbf{n}), \quad {\rm with }\, \lim_{m\to \infty}\frac{1}{m}=0.
\end{equation*}		
For the case where $n_1$ is odd, the proof is of the same type as before with a few
slight technical changes in the decomposition in (\ref{eq:dec_hank_tep}). Indeed,  the result of Theorem \ref{Lemm:symm} is maintained also for odd dimensions.  Then, is possible to follow an analogous proof  as it has been done for the unilevel case in \cite[Theorem 3.2]{FFHMS} 
\end{proof}

	\begin{theorem}\label{thm:main_L1}
	Suppose $f\in L^1([-\pi,\pi]^k)$ a $k$-variate function  with real Fourier coefficients,  periodically extended to the whole real plane. Let $Y_\mathbf{n} \in \mathbb{R}^{N(\mathbf{n}) \times N(\mathbf{n})}$ be the anti-identity matrix $Y_\mathbf{n}=Y_{n_1}\otimes \ldots \otimes Y_{n_k}=Y_{N(\mathbf{n})}$ and let $T_\mathbf{n}[f]\in \mathbb{R}^{N(\mathbf{n}) \times N(\mathbf{n})}$ be the Toeplitz matrix generated by $f$. Then,
		\[
		\{Y_\mathbf{n}T_\mathbf{n}[f]\}_{\mathbf{n}}  \sim_{\lambda}  \psi_{|f|}
		\]
		over the domain $[-2\pi,0]^k\cup [0,2\pi]^k$, where $\psi_{|f|}$ is given as in Definition \ref{def:psi}.
	\end{theorem}
	\begin{proof}
	Since the set of $k$-th variate trigonometric polynomials is dense in $L^1([-\pi,\pi]^k)$, there exists a sequence $\{f_m\}_m$ of trigonometric polynomials such that $\{f_m\}_m\to f$
in $L^1([-\pi,\pi]^k)$. 

From Lemma \ref{lem:Corollary5.1}, we obtain that $\{Y_\mathbf{n}T_\mathbf{n}[f]\}_{\mathbf{n}}  \sim_{\lambda}  \psi_{|f|}$ if the sequences $\{Y_\mathbf{n}T_\mathbf{n}[f]\}_{\mathbf{n}}  $, $\{Y_\mathbf{n}T_\mathbf{n}[f_m]\}_{\mathbf{n}}$ and the functions $ \psi_{|f|}$ and $ \psi_{|f_m|}$ satisfy the following:
\begin{itemize}
\item  $\{Y_\mathbf{n}T_\mathbf{n}[f_m]\}_{\mathbf{n}}\sim_{\lambda}  \psi_{|f_m|}$;
\item $\{Y_{\mathbf{n}}T_{\mathbf{n}}[f_m]\}_{\mathbf{n}}\xrightarrow{\text{a.c.s.\ wrt\ $m$}}\{Y_{\mathbf{n}}T_\mathbf{n}(f)\}_{\mathbf{n}};$
\item $ \{\psi_{|f_m|}\}_m\to \psi_{|f|}$ in measure.
\end{itemize} 
The first item is a consequence of Lemma \ref{thm:main_pol}, since each $f_m$ is a multivariate trigonometric polynomial.

For the second item we use the characterization Theorem \ref{thm:acs_caratt}.
In particular, it holds
\begin{equation*}
\|Y_{\mathbf{n}}T_\mathbf{n}(f)- Y_{\mathbf{n}}T_\mathbf{n}(f_m)\|_1\le \|Y_{\mathbf{n}}\| \|T_\mathbf{n}(f-f_m)\|_1\le \frac{N(\mathbf{n})}{(2\pi)^k}\|f-f_m\|_{L^1}.
\end{equation*}
Since $\lim_{m\to \infty} \|f-f_m\|_{L^1}=0,$ we can conclude that $$\{Y_{\mathbf{n}}T_{\mathbf{n}}[f_m]\}_{\mathbf{n}}\xrightarrow{\text{a.c.s.\ wrt\ $m$}}\{Y_{\mathbf{n}}T_\mathbf{n}(f)\}_{\mathbf{n}}.$$
	
Finally, we prove that $ \{\psi_{|f_m|}\}_m\to \psi_{|f|}$ in $L^1([-\pi,\pi]^k)$, so that the convergence in measure is a direct consequence.

In particular, we have to prove that
\begin{equation*}
\lim_{m\to \infty}\int_{\Omega} |\psi_{|f|}-\psi_{|f_m|}| \,d\boldsymbol{\theta}=0,
\end{equation*}
where $\Omega=[-2\pi,0]^k\cup[0,2\pi]^k$.
Since $f$ and $f_m$ are $2\pi$-periodic functions, it holds

\begin{equation}
\begin{split}
\int_{\Omega} |\psi_{|f|}-\psi_{|f_m|}| \,d\boldsymbol{\theta}=&\int_{(0,2\pi]^k}|{|f|}-{|f_m|}| \,d\boldsymbol{\theta}+ \int_{[-2\pi,0)^k}|{-|f|}+{|f_m|}|  \,d\boldsymbol{\theta}=\\
&2\int_{[-\pi,\pi]^k}|{|f|}-{|f_m|}|  \,d\boldsymbol{\theta}\le 2 \int_{[-\pi,\pi]^k}|{f}-{f_m}|  \,d\boldsymbol{\theta}.
\end{split}
\end{equation}
Then, since  $\{f_m\}_m\to f$
in $L^1([-\pi,\pi]^k)$,

\begin{equation*}
\lim_{m\to \infty}\int_{\Omega} |\psi_{|f|}-\psi_{|f_m|}| \,d\boldsymbol{\theta}\le \lim_{m\to \infty} 2 \int_{[-\pi,\pi]^k}|{f}-{f_m}|  \,d\boldsymbol{\theta}=0,
\end{equation*}
which implies that $ \{\psi_{|f_m|}\}_m\to \psi_{|f|}$ in measure.

		\end{proof}

		Note that $\psi_{|f|}$ is not the unique function which describes the asymptotic spectral distribution of $\{Y_{\mathbf{n}}T_{\mathbf{n}}[f]\}_{\mathbf{n}}$. Indeed, it is possible to find a rearrangement  $\phi_{|f|}$ of $\psi_{|f|}$ 
(and viceversa) such that $\{Y_{\mathbf{n}}T_{\mathbf{n}}[f]\}_{\mathbf{n}} \sim_{\lambda} \phi_{|f|},$ see Corollary \ref{thm:main_phi}. In addition, under the hypotheses of separability of $f$ we can construct a spectral symbol $h_{f}$ with a tensor-product argument, see Proposition \ref{prop:separability}.

\begin{corollary}\label{thm:main_phi}
		Suppose $f \in L^1([-\pi,\pi]^k)$ is a $k$-variate function  with real Fourier coefficients,  periodically extended to the whole real plane. Let $Y_\mathbf{n}=Y_{n_1}\otimes \ldots \otimes Y_{n_k}=Y_{N(\mathbf{n})} \in \mathbb{R}^{N(\mathbf{n}) \times N(\mathbf{n})}$ be the anti-identity matrix. Let $T_\mathbf{n}[f]\in \mathbb{R}^{N(\mathbf{n}) \times N(\mathbf{n})}$ be the Toeplitz matrix generated by $f$. Then, 
		\[
		\{Y_{\mathbf{n}}T_{\mathbf{n}}[f]\}_{\mathbf{n}} \sim_{\lambda}  \phi_{|f|}
		\]
		over the domain  $[-2\pi,2\pi]^k$ with $\phi_g$ defined in the following way
		\[
		\phi_g(\boldsymbol{\theta})=\left\{
		\begin{array}{cc}
		g(\boldsymbol{\theta}), & \boldsymbol{\theta}\in [0,2\pi]^k, \\
		-g(-\boldsymbol{\theta}), & \boldsymbol{\theta}\in [-2\pi,0)^k.
		\end{array}
		\right.\,
		\]

	\begin{proof}
		We observe that for any $F$ continuous with bounded support
		\[
		\int_{[-2\pi,2\pi]^k} F(\phi_{|f|})=\int_{[-2\pi,2\pi]^k}  F(\psi_{|f|}),
		\]
		i.e. $\phi_{|f|}$ is a rearrangement of $\psi_{|f|}$ (and viceversa) \cite[Section 3.2]{SerraLibro1}. Hence, by the very definition of distribution, we have
		$\{Y_\mathbf{n} T_\mathbf{n}[f]\}_{\mathbf{n}} \sim_{\lambda}  \phi_{|f|}$ if and only if $\{Y_\mathbf{n} T_\mathbf{n}[f]\}_{\mathbf{n}} \sim_{\lambda}  \psi_{|f|}$. Therefore, the desired result is an immediate consequence of Theorem \ref{thm:main_L1}. 
	\end{proof}
\end{corollary}

\section{Further results and remarks}\label{section:further}

The distribution results can be combined and complemented with analogous studies on the preconditioned matrix-sequences: in this sense the paper by Pestana already is a step in this direction, that is in the multilevel setting (refer to \cite{PestanaSIMAX}).

Furthermore, the proof techniques employed so far can be easily extended to the case of a generating function which is multivariate and matrix-valued, so covering the multilevel block setting as done in \cite{FFHMS} for the univariate case.

The case of separable generating functions deserves particular attention because one has beautiful tensor structures and hence the distributional results in the multilevel context can be directly deduced by the unilevel case, as shown in some detail in the next result. 

\begin{proposition}\label{prop:separability}
Suppose $f \in L^1([-\pi,\pi]^k)$ is a $k$-variate function periodically extended to the whole real plane. Assume that $f$ is $k$-separable, that is, there exist $k$ functions $f_i\in L^1([-\pi,\pi])$, $i=1,\dots,k$, such that $f = f_1 \otimes \dots \otimes f_k$. Suppose that the functions $f_i$, $i=1,\dots,k$, have real Fourier coefficients. Let $Y_\mathbf{n}=Y_{n_1}\otimes \ldots \otimes Y_{n_k}=Y_{N(\mathbf{n})} \in \mathbb{R}^{N(\mathbf{n}) \times N(\mathbf{n})}$ be the anti-identity matrix. Let $T_\mathbf{n}[f]\in \mathbb{R}^{N(\mathbf{n}) \times N(\mathbf{n})}$ be the Toeplitz matrix generated by $f$. Then, 
		\[
		\{Y_{\mathbf{n}}T_{\mathbf{n}}[f]\}_{\mathbf{n}} \sim_{\lambda}  h_{f}
		\]
		over the domain  $[-2\pi,2\pi]^k$, where the function $h_f$ is defined as
		\begin{equation}\label{eq:separability}
			h_f=\psi_{|f_1|}\otimes\dots \otimes\psi_{|f_k|}
		\end{equation}
and $\psi_{|f_i|}$, $i=1,\dots,k$, is defined as in equation (\ref{def:psi_unilevel}).
Furthermore, $\psi_{|f|}$, $\phi_{|f|}$ are both rearrangements of $h_f$ (and viceversa).
\end{proposition}
\begin{proof}
From Lemma \ref{lemm:tensor_prod} and the definition of $Y_{\mathbf{n}}$, we obtain
\begin{equation*}
Y_{\mathbf{n}}T_{\mathbf{n}}[f]= Y_{\mathbf{n}}T_{\mathbf{n}} [f_1 \otimes \dots \otimes f_k ]=(Y_{{n}_1}\otimes\dots \otimes Y_{{n}_k})(T_{ n_1} [ f_1 ] \otimes \dots \otimes T_{n_k} [ f_k ]).
\end{equation*}
From the mixed product property of the Kronecker product, we have
\begin{equation*}
Y_{\mathbf{n}}T_{\mathbf{n}}[f]=Y_{n_1}T_{n_1}[f_1]\otimes \dots \otimes Y_{n_k}T_{n_k}[f_k].
\end{equation*}
From Theorem \ref{thm:main_old1d}, the following spectral asymptotic results hold for $j=1,\dots,k$
\[\{Y_{n_j}T_{n_j}[f_j]\}_{n_j}\sim_{\lambda}\psi_{|f_i|}.\]
Hence, the spectrum of $\{Y_{\mathbf{n}}T_{\mathbf{n}}[f]\}_{\mathbf{n}}$ is asymptotically described by the tensor products of each  $\psi_{|f_i|}$, $i=1,\dots,k$. That is,

\begin{equation*}
\{Y_{\mathbf{n}}T_{\mathbf{n}}[f]\}_{\mathbf{n}} \sim_{\lambda}\psi_{|f_1|}\otimes\dots \otimes\psi_{|f_k|}= h_f.
\end{equation*}
Finally, by combining the last equation with Theorem \ref{thm:main_L1} and Corollary \ref{thm:main_phi}, by virtue of the definition of spectral distribution, we deduce that the functions $\psi_{|f|}$, $\phi_{|f|}$ are necessarily both rearrangements of $h_f$ (and viceversa).
\end{proof}

Note that a $k$-variate trigonometric polynomial of degree $\mathbf{r}=(r_1,r_2,\dots,r_k)$ with real Fourier coefficients $\hat{f}_{\mathbf{j}}$ has the explicit form
\begin{equation}\label{eq:pol_stability}
f(\boldsymbol{\theta})=\sum_{{j_1}=-{r_1}}^{r_1}\sum_{{j_2}=-{r_2}}^{r_2}\dots \sum_{{j_k}=-{r_k}}^{r_k}\hat{f}_{j_1,j_2,\dots,j_k}\E^{\iota\sum_{i=1}^kj_i\theta_i}.		
\end{equation}
That is, it can be written as finite sums of separable trigonometric polynomial $\{\hat{f}_{j_1,j_2,1\dots,j_k}\E^{\iota\sum_{i=1}^kj_i\theta_i}\}_{j_1,j_2,\dots,j_k}$.
Hence, on each term we can apply Lemma \ref{lemm:tensor_prod} and obtain, for $\mathbf{j}=-\mathbf{r},\dots,\mathbf{r}$,

\begin{equation*}
\left\{Y_{\mathbf{n}}T_{\mathbf{n}}\left[\hat{f}_{j_1,j_2,\dots,j_k}\E^{\iota\sum_{i=1}^kj_i\theta_i}\right]\right\}_{\mathbf{n}} \sim_{\lambda} \hat{f}_{j_1,j_2,1\dots,j_k}\left(\psi_{|\E^{\iota j_1\theta_1}|}\otimes\dots\otimes \psi_{|\E^{\iota j_k\theta_k}|}\right).
\end{equation*}
Therefore, an interesting further investigation is the study of the stability of the spectral distribution for the sum in (\ref{eq:pol_stability}). In fact, such a proof would lead to a simplification when treating distribution results for generic multilevel matrix-sequences.
 Finally, it is worth mentioning that all the results in Theorem \ref{thm:main_L1}, Corollary \ref{thm:main_phi}, Proposition \ref{prop:separability} also hold in the version for singular values for a very basic reason, so that 

\begin{equation}\label{3 sv results} 
\{Y_{\mathbf{n}}T_{\mathbf{n}}[f]\}_{\mathbf{n}} \sim_{\sigma}\psi_{|f|}, \ \ \  
\{Y_{\mathbf{n}}T_{\mathbf{n}}[f]\}_{\mathbf{n}} \sim_{\sigma}\phi_{|f|}, \ \ \  
\{Y_{\mathbf{n}}T_{\mathbf{n}}[f]\}_{\mathbf{n}} \sim_{\sigma} h_f, 
\end{equation} 

respectively. In reality, since $Y_{\mathbf{n}}T_{\mathbf{n}}[f]$ is real symmetric for every $\mathbf{n}$, by comparing the singular value decomposition and the Schur normal form, it is immediate to see that the singular values of $Y_{\mathbf{n}}T_{\mathbf{n}}[f]$ are the modulus of corresponding eigenvalues. As consequence, taking into account Theorem \ref{thm:singularvalue_multi}, we infer that $|\psi_{|f|}|$, $|\phi_{|f|}|$, $|h_f|$ are all rearrangements of $|f|$.

	\section{Numerical experiments}\label{section:numerics}
	
 In this section we numerically show that the statements of Theorem~\ref{thm:main_L1} and Proposition~\ref{prop:separability} are true in the cases of both  bivariate trigonometric polynomials and generic separable functions in $L^1([-\pi,\pi]^2)$. In particular we illustrate the predicted behaviour of the eigenvalues for the matrix-sequences $\{Y_{\mathbf{n}} T_{\mathbf{n}}[f]\}_{\mathbf{n}} $ for a function $f$ in the following cases. 
 \begin{itemize}
 \item Example 1.  $f$ is a bivariate trigonometric polynomial $f:[-\pi,\pi]^2\mapsto \mathbb{C}$ with high degree.
 \item Example 2.  $f$ is the bivariate separable trigonometric polynomial $f:[-\pi,\pi]^2\mapsto \mathbb{C}$, $$
		f(\theta_1,\theta_2) = \left(10-3{\rm e}^{\iota(\theta_1)}+{\rm e}^{\iota(-\theta_1)}\right)
							   \left(4-3{\rm e}^{\iota(\theta_2)}\right).$$
 \item Example 3. $f\in L^1([-\pi,\pi]^2)$ is the separable function
		$f(\theta_1,\theta_2) = \theta_1^2\theta_2^2$.
 \end{itemize}

	In order to numerically support the validity of the asymptotic spectral distribution provided by Theorem \ref{thm:main_L1} (resp. Proposition \ref{prop:separability}), we show that for large enough $n$ the eigenvalues of $Y_\mathbf{n}T_\mathbf{n}[f]$ are approximately equal to the samples of $\psi_{|f|}$ (resp. $h_f$) over a uniform grid. We remark that the theory admits the possible exception of outliers, whose number is infinitesimal respect to the dimension $N(\mathbf{n})$ of the matrix.  
	
Moreover, since the matrices $Y_\mathbf{n} T_\mathbf{n}[f]$ are symmetric for any $\mathbf{n} $, the values $\lambda_j(Y_\mathbf{n} T_\mathbf{n}[f])$ are real for $j=1,\dots,N(\mathbf{n})$. Then we can give an order to them according to the evaluation of $\psi_{|f|}$ (resp. $h_f$) on the proper grid.

	\begin{example}\label{ex_polyH}
		We consider the bivariate trigonometric polynomial $f:[-\pi,\pi]^2\mapsto \mathbb{C}$ defined by
		\begin{align}\label{eq:polyH}
				f(\theta_1,\theta_2) = &5+{\rm e}^{\iota(\theta_1)}+{\rm e}^{\iota(2\theta_1)}-3{\rm e}^{\iota(-\theta_1)}+2{\rm e}^{\iota(-2\theta_1)}
		+2{\rm e}^{\iota(\theta_2)}+{\rm e}^{\iota(2\theta_2)}-2{\rm e}^{\iota(-\theta_2)}+{\rm e}^{\iota(-2\theta_2)} \\ \nonumber
		&+{\rm e}^{\iota(\theta_1+\theta_2)}+3{\rm e}^{\iota(\theta_1-\theta_2)}+3{\rm e}^{\iota(-\theta_1+\theta_2)}+2{\rm e}^{\iota(-\theta_1-\theta_2)}
		+{\rm e}^{\iota(-2\theta_1+2\theta_2)}+{\rm e}^{\iota(\theta_1+2\theta_2)}+2{\rm e}^{\iota(-2\theta_1+\theta_2)}\\ \nonumber
		&+3{\rm e}^{\iota(2\theta_1+\theta_2)}
		+{\rm e}^{\iota(-2\theta_1-\theta_2)}+{\rm e}^{\iota(2\theta_1-\theta_2)}+{\rm e}^{\iota(-2\theta_1-2\theta_2)}+{\rm e}^{\iota(-\theta_1-2\theta_2)}
		+{\rm e}^{\iota(\theta_1-2\theta_2)}+{\rm e}^{\iota(2\theta_1-2\theta_2)}.
		\end{align}
		Hence $f$ is a trigonometric polynomial with real Fourier coefficients. In particular, the coefficients of $f$ can be represented in a more compact and elegant form as a $2D$ stencil
		\begin{equation}\label{eq:stencil_1}
		\begin{bmatrix}
		 1   &  1 &    1&     3 &    0\\
     1   & 3  &  1 &    1 &   1\\
     1     &-2   &  5  &  2   &  1\\
     1   &  2    & -3   &  3   & 0\\
     1    & 1   & 2    & 2     &1
		\end{bmatrix}.
		\end{equation}

	Hence, Theorem \ref{thm:main_L1} implies that the eigenvalues of $Y_\mathbf{n} T_\mathbf{n}[f]$ (properly sorted) are approximately equal to the samples of $\psi_{|f|}$ over the following grid $\xi_{j_1,j_2}^{(n_1,n_2)}$ of the domain $[-2\pi,0]^2\cup[0,2\pi]^2$

		\begin{equation}\label{eq:union_grid}
		\xi_{j_1,j_2}^{(n_1,n_1)} = \gamma_{j_1,j_2}^{(n_1,n_2)} \cup (\gamma_{j_1,j_2}^{(n_1,n_2)}-2\pi),
		\end{equation}
		where
		\begin{equation*}
	\gamma_{j_1,j_2}^{(n_1,n_2)}=2\pi \frac{(j_1,j_2)}{\left(\frac{n_1}{2}-1,n_2-1\right)}, \quad j_1=0,\dots,\frac{n_1}{2}-1, \quad j_2=0,\dots,{n_2}-1.
	\end{equation*}		
		
		We fix  $n_1=n_2=32$. Then, $Y_\mathbf{n} T_\mathbf{n}[f]\in \mathbb{R}^{N(\mathbf{n})\times N(\mathbf{n})}$ is a  $N(\mathbf{n})\times N(\mathbf{n})$ real symmetric matrix with eigenvalues $\lambda_{j}(Y_\mathbf{n} T_\mathbf{n}[f])$, $j=1,\dots,N(\mathbf{n})$ and $N(\mathbf{n})=1024$.
	In Figure \ref{fig:ex_polyH} we plot (in gray) the samples of $\psi_{|f|}$ over the grid $\xi_{j_1,j_2}^{(32,32)}$ and we can observe that they approximate without outliers the eigenvalues (colored dots) $\lambda_j(Y_{(32,32)} T_{(32,32)}[f])$, $j=1,\dots, 1024$ sorted according to the order given by $\psi_{|f|}$ over $\xi_{j_1,j_2}^{(32,32)}$.

		\begin{figure}
			\centering
			\includegraphics[width=\textwidth]{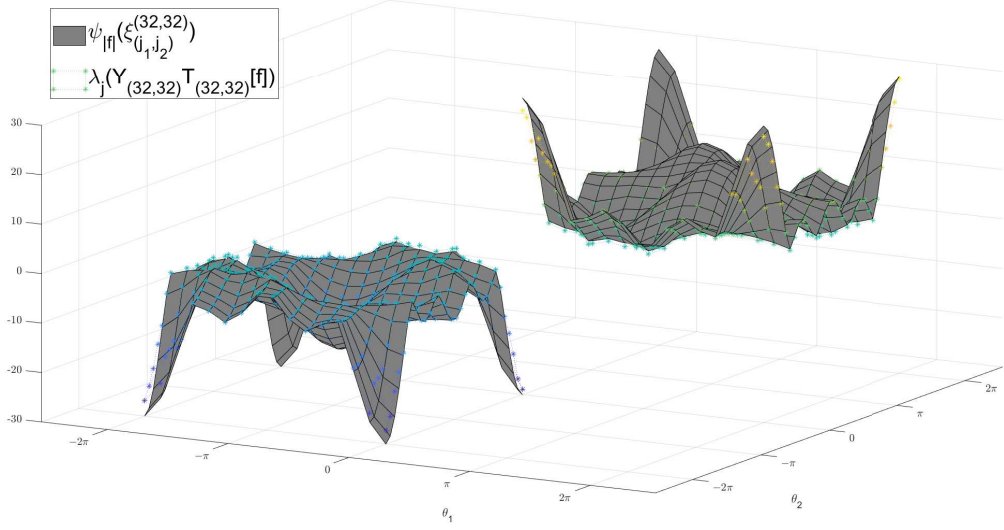}
			\vspace{-10pt}
			\caption{Example~\ref{fig:ex_polyH}, a comparison between the eigenvalues  $\lambda_j(Y_{(32,32)} T_{(32,32)}[f])$ and the samples $\psi_{|f|}$ over the grid $\xi_{(j_1,j_2)}^{(32,32)}$, where $f$ is defined by (\ref{eq:polyH}).
			}
			\label{fig:ex_polyH}
		\end{figure}
\end{example}
	
\begin{example}\label{ex_poly2}
		We consider the bivariate separable trigonometric polynomial $f:[-\pi,\pi]^2\mapsto \mathbb{C}$
	\begin{equation}\label{eq:ex_poly2}
		f(\theta_1,\theta_2) = \left(10-3{\rm e}^{\iota(\theta_1)}+{\rm e}^{\iota(-\theta_1)}\right)
							   \left(4-3{\rm e}^{\iota(\theta_2)}\right).
		\end{equation}
		In particular $f(\theta_1,\theta_2)$ can be decomposed, analogously  as $f$ of  Proposition \ref{prop:separability}, as the tensor of univariate polynomials $f_1(\theta_1)=10-3{\rm e}^{\iota(\theta_1)}+{\rm e}^{\iota(-\theta_1)}$ and $f_2(\theta_2)=4-3{\rm e}^{\iota(\theta_2)}$. Moreover, it has the stencil given by

		\begin{equation}\label{eq:stencil_2}
		\begin{bmatrix}
	0& -12& -3\\
       0& 40& -30\\
     0& 4&9\\
		\end{bmatrix}=\begin{bmatrix}
		-3\\
		10\\
		1
		\end{bmatrix}\begin{bmatrix}
		0& 4 & -3
		\end{bmatrix}.
		\end{equation}
	
This implies that we can provide two equivalent spectral distributions for the sequence $\{Y_\mathbf{n} T_\mathbf{n}[f]\}_{\mathbf{n}} $. Indeed, on the one hand, $f$ is a trigonometric polynomial with real Fourier coefficients, then, from Theorem \ref{thm:main_L1}
\begin{equation*}
\{Y_\mathbf{n} T_\mathbf{n}[f]\}_{\mathbf{n}} \sim_{\lambda} \psi_{|f|},
\end{equation*}
where $\psi_{|f|}$ is defined as in formula (\ref{eq:psi}). On the other hand, $f$ is a separable trigonometric polynomial satisfying the hypotheses of Proposition \ref{prop:separability}. Then,

\begin{equation*}
\{Y_\mathbf{n} T_\mathbf{n}[f]\}_{\mathbf{n}} \sim_{\lambda} h_f,
\end{equation*}  
   where $h_f$ is defined as in (\ref{eq:separability}).
   
   In Figure \ref{fig:ex_poly_sep} we numerically show that both the results are accurate in approximating the eigenvalues of $Y_\mathbf{n} T_\mathbf{n}[f]$, since $\psi_{|f|}$ is a rearrangement of $h_f$ and viceversa. 
  In particular, on the top of Figure \ref{fig:ex_poly_sep}, we can observe that the  eigenvalues $\lambda_j(Y_\mathbf{n} T_\mathbf{n}[f])$, $j=1,\dots, N(\mathbf{n})$, $\mathbf{n}=(32,32)$, are well approximated by the samplings of the function $\psi_{|f|}$ over the grid $\xi_{(j_1,j_2)}^{(n_1,n_2)}$. The grid $\xi_{(j_1,j_2)}^{(n_1,n_2)}$ of the domain $[-2\pi,0]^2\cup[0,2\pi]^2$, is defined as in (\ref{eq:union_grid}).
  Analogously, the bottom of Figure \ref{fig:ex_poly_sep} numerically confirms that an accurate approximation of the eigenvalues  $\lambda_j(Y_\mathbf{n} T_\mathbf{n}[f])$, $j=1,\dots, N(\mathbf{n})$, $\mathbf{n}=(32,32)$, can be computed by a uniform sampling of the function $h_f=\psi_{|f_1|}\otimes\psi_{|f_2|}$ over the domain $[-2\pi,2\pi]^2$. The uniform grid over  $[-2\pi,2\pi]^2$ is $\theta_{(j_1,j_2)}^{(32,32)}$ where
  \begin{equation}\label{eq:uniform_grid}
	\theta_{(j_1,j_2)}^{(n_1,n_2)}=4\pi \frac{(j_1,j_2)}{\left(n_1-1,n_2-1\right)}-2\pi, \quad j_1=0,\dots,n_1-1, \quad j_2=0,\dots,n_2-1.
\end{equation}   

		\begin{figure}
			\centering
			\includegraphics[width=\textwidth]{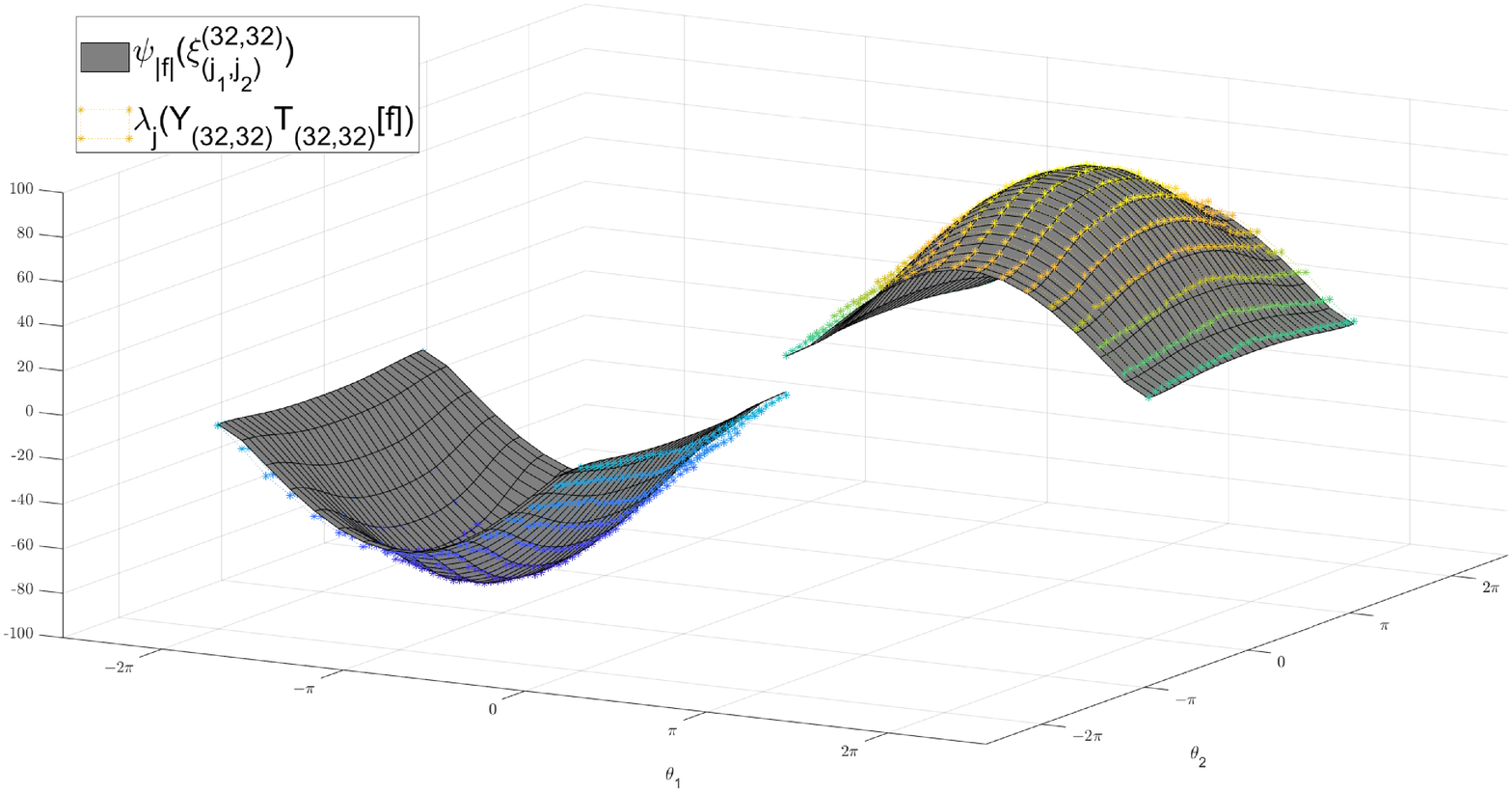}
			\includegraphics[width=\textwidth]{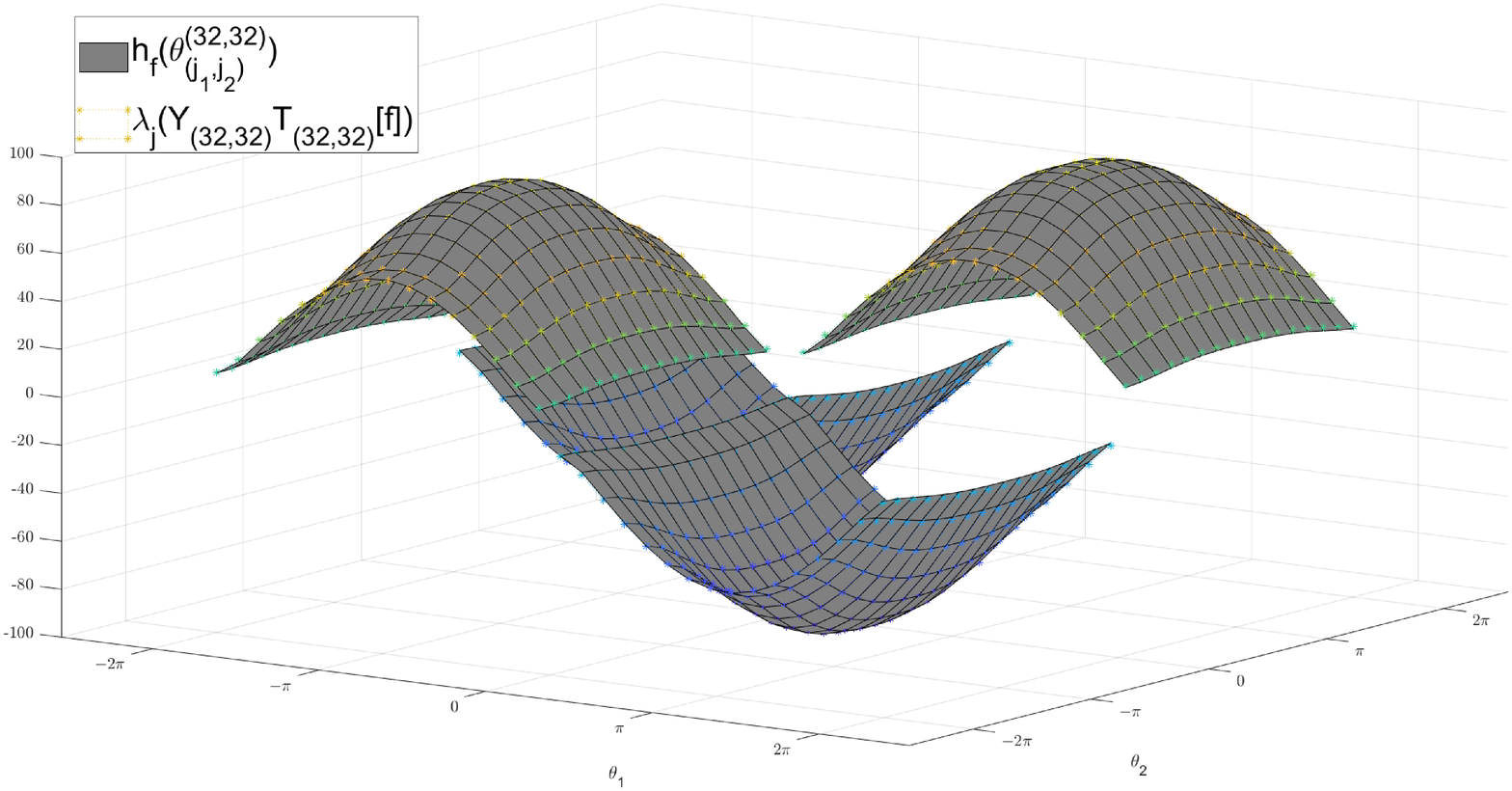}
			\vspace{-10pt}
			\caption{Example~\ref{ex_poly2},   a comparison between the eigenvalues  $\lambda_j(Y_{(32,32)} T_{(32,32)}[f])$ and  (top) the samples $\psi_{|f|}$ over the grid $\xi_{(j_1,j_2)}^{(32,32)}$ and (bottom) a uniform sampling of the function $h_f$ over the grid $\theta_{(j_1,j_2)}^{(32,32)}$. The function $f$ is defined by (\ref{eq:ex_poly2}).
			}
			\label{fig:ex_poly_sep}
		\end{figure}
		
	\end{example}
	
	\begin{example}\label{ex_teta2}
		In the last example we consider the function $f:[-\pi,\pi]^2 \rightarrow \mathbb{R}$ by
		\begin{equation}\label{eq:ex_teta2}
			f(\theta_1,\theta_2) = \theta_1^2\theta_2^2,	
		\end{equation}
	
		periodically extended to the real plane.
		
		The function $f$ is not a trigonometric polynomial, and consequently the matrices $T_\mathbf{n}[f]$ are dense, \textcolor{black}{for all $\mathbf{n}$}.  Note that $f$ is a bivariate separable function in $L^1([-2\pi,2\pi]^2)$ obtained as the tensor of univariate functions with real Fourier coefficients.
		
In particular the Fourier coefficients of the univariate function $\theta^2$ are given by the \textcolor{black}{formulae}
		\[
		\left\{
		\begin{array}{ll}
		a_0 = \frac{\pi^2}{3}, &  \\
		a_{k}=(-1)^k\frac{2}{k^2}, & k=\pm1,\pm2,\dots.
		\end{array}
		\right. .
		\]
	
Recalling that $f$ is defined on $[-\pi,\pi]^2$ and periodically extended to the real plane, we can write the following explicit formulae for $f$ in $[0,2\pi]^2$:
\[
		\left\{
		\begin{array}{ll}
		\theta_1^2\theta_2^2, & (\theta_1,\theta_2)\in [0,\pi]^2, \\
		\theta_1^2(\theta_2-2\pi)^2, & (\theta_1,\theta_2)\in [0,\pi]\times(\pi,2\pi],\\
		(\theta_1-2\pi)^2\theta_2^2, & (\theta_1,\theta_2)\in (\pi,2\pi]\times[0,\pi],\\
		(\theta_1-2\pi)^2(\theta_2-2\pi)^2, & (\theta_1,\theta_2)\in (\pi,2\pi]^2.
		\end{array}
		\right. 
\]	
		
Since $f$ is a function in $L^1([-\pi,\pi]^2)$ with real Fourier coefficients, then, from Theorem \ref{thm:main_L1}, we have the following spectral distribution result:
\begin{equation*}
\{Y_\mathbf{n} T_\mathbf{n}[f]\}_{\mathbf{n}} \sim_{\lambda} \psi_{|f|},
\end{equation*}
where $\psi_{|f|}$ is defined as in formula (\ref{eq:psi}). Moreover, $f$ is a separable function satisfying the hypotheses of Proposition \ref{prop:separability}, and so  also the following spectral distribution holds:
\begin{equation*}
\{Y_\mathbf{n} T_\mathbf{n}[f]\}_{\mathbf{n}} \sim_{\lambda} h_f,
\end{equation*}  
where $h_f$ is defined as in (\ref{eq:separability}).		
		
In Figure \ref{fig:ex_teta2} we numerically show that both the results are accurate in approximating the eigenvalues of $Y_\mathbf{n} T_\mathbf{n}[f]$. 
  In particular, on the top of Figure \ref{fig:ex_teta2}, we can observe that the  eigenvalues $\lambda_j(Y_\mathbf{n} T_\mathbf{n}[f])$, $j=1,\dots, N(\mathbf{n})$, $\mathbf{n}=(64,64)$, are well approximated by the samplings of the function $\psi_{|f|}$ over the grid $\xi_{(j_1,j_2)}^{(n_1,n_2)}$ defined in (\ref{eq:union_grid}).
  Analogously, the bottom of Figure \ref{fig:ex_teta2} numerically confirms that an accurate approximation of the eigenvalues  $\lambda_j(Y_\mathbf{n} T_\mathbf{n}[f])$, $j=1,\dots, N(\mathbf{n})$, $\mathbf{n}=(64,64)$, can be computed by a uniform sampling of the function $h_f=\psi_{|f_1|}\otimes\psi_{|f_2|}$ over the domain $[-2\pi,2\pi]^2$.

		\begin{figure}
			\centering
			\includegraphics[width=\textwidth]{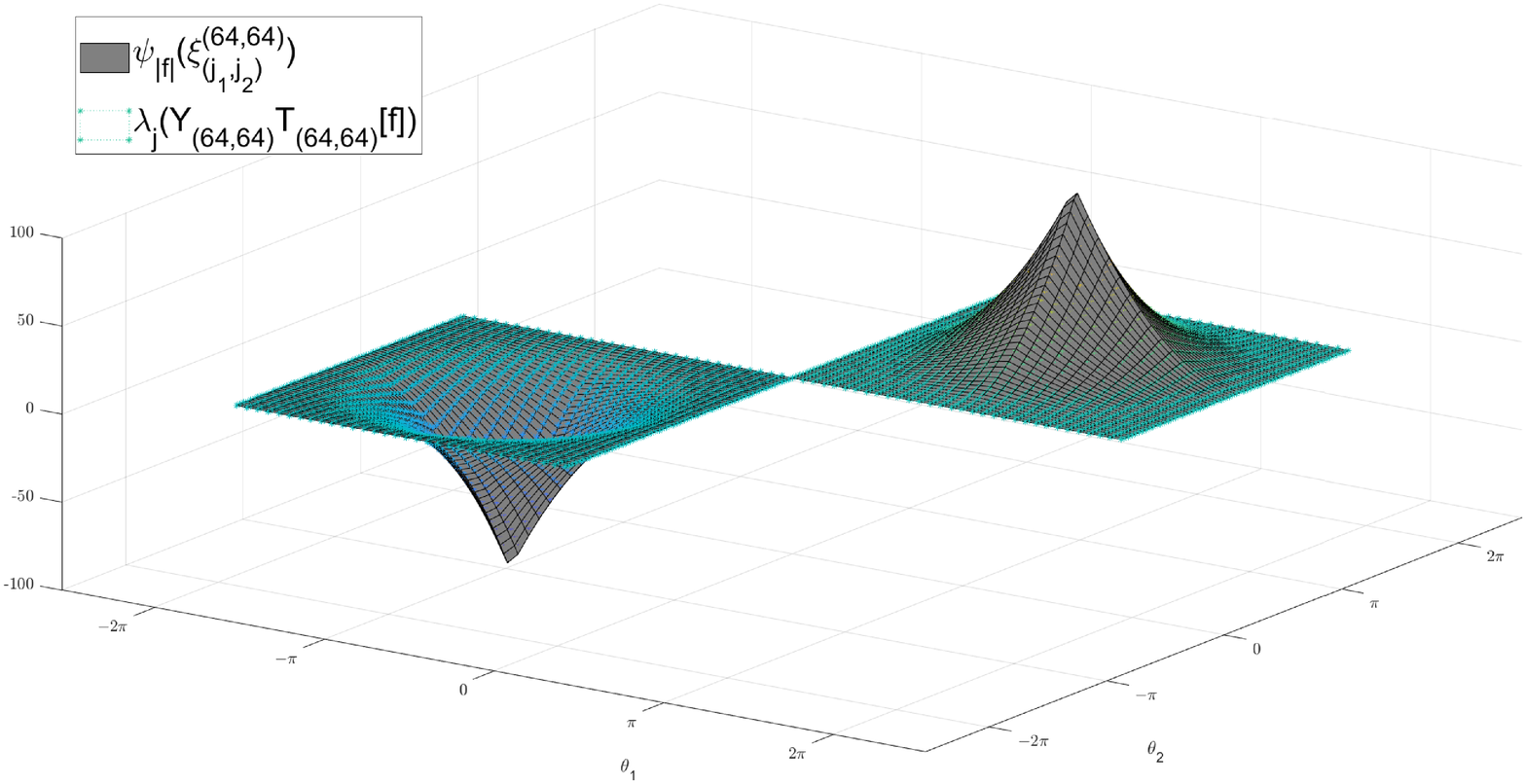}
			\includegraphics[width=\textwidth]{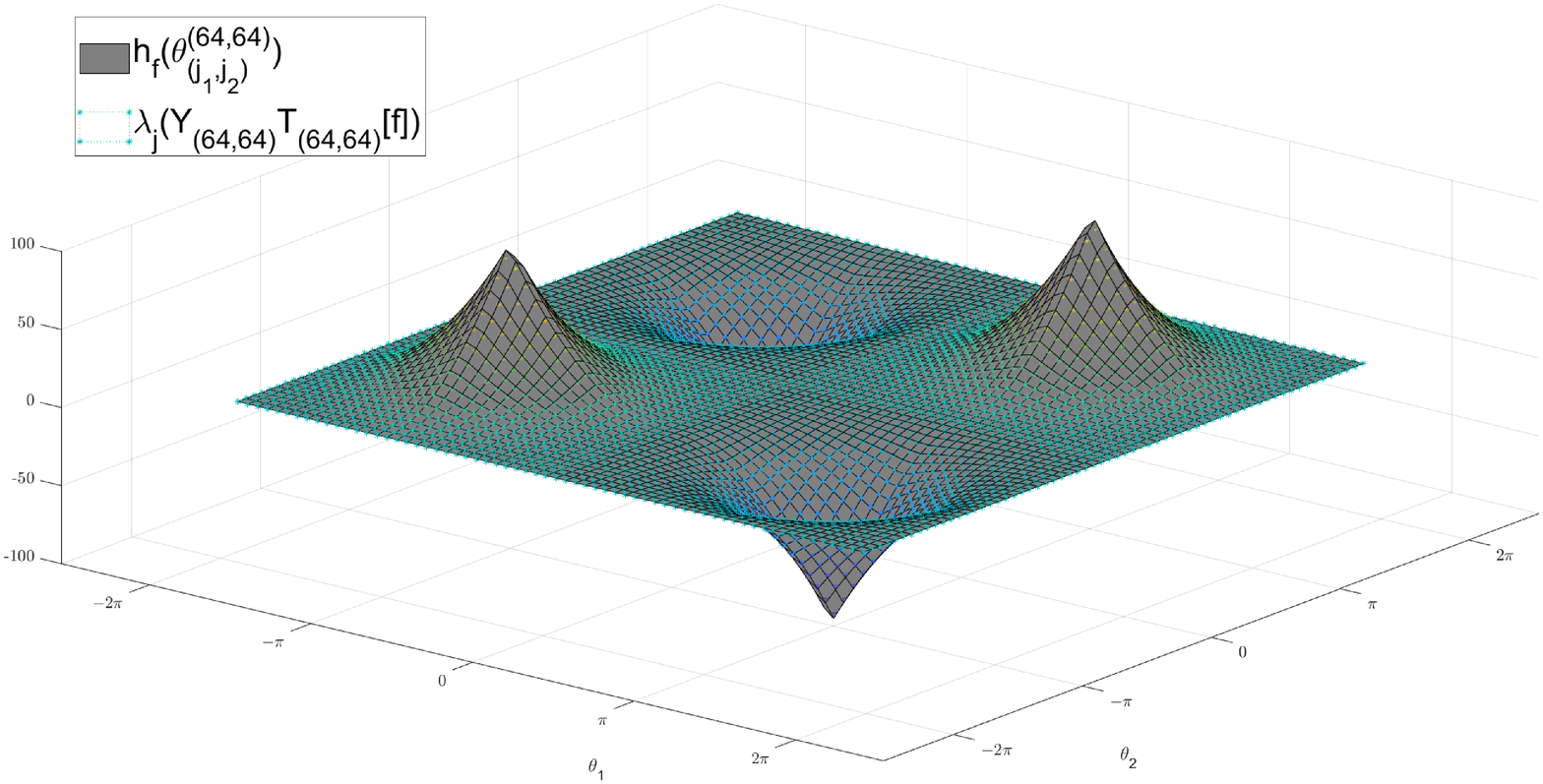}
			\vspace{-10pt}
			\caption{Example~\ref{ex_teta2},   a comparison between the eigenvalues  $\lambda_j(Y_{(64,64)} T_{(64,64)}[f])$ and  (top) the samples $\psi_{|f|}$ over the grid $\xi_{(j_1,j_2)}^{(64,64)}$ and (bottom) a uniform sampling of the function $h_f$ over the grid $\theta_{(j_1,j_2)}^{(64,64)}$. The function $f$ is defined by (\ref{eq:ex_teta2}).
			}
			\label{fig:ex_teta2}
		\end{figure}
	\end{example}
	
	\section{Conclusions}\label{sec:conclusion}
	
	We have extended the asymptotic results concerning the singular  value and spectral distribution for the symmetrization of multilevel Toeplitz-sequence.  
	Taking inspiration from the unilevel matrix-sequence $\{Y_nT_n[f]\}_n$, we focused of the case where $f$ is a $k$-variate function in $ L^1([-\pi,\pi])^k$ and  $\{Y_{\mathbf{n}}T_{\mathbf{n}}[f]\}_{\mathbf{n}} $ the symmetrization of a $k$-level Toeplitz-sequence. We analysed the case where $f$ is a trigonometric polynomial, and then we extend the result to $f\in L^1([-\pi,\pi]^k)$, $k>1$, by the notion and properties of the approximating class of sequences. In addition, we dedicated particular attention to the case where $f$ is a $k$-variate separable function. Under this additional hypothesis, the spectral distribution of the sequence $\{Y_{\mathbf{n}}T_{\mathbf{n}}[f]\}_{\mathbf{n}} $ can be obtained by a tensor product argument from the 1D setting. The study of possible simplifications when dealing with the class of polynomials will be
subject of the future investigation. Moreover, as a future goal we intend to exploit  the derived  knowledge of the asymptotic spectral and singular value distribution in order to design efficient solvers for large linear system stemming from practical applications and in this direction a first step is done in \cite{PestanaSIMAX,mazza-pestana-m}.

\bigskip
{\bf Acknowledgment.}  Paola Ferrari, Isabella Furci, and Stefano Serra-Capizzano are partially supported by the INdAM Research group GNCS.

%%%%%%%%%%%%%%%%%%%%%%%%%%%%%%%%%%%%%%%%%%%%%%%%%%%%%%%%%%%%%

\end{document}